\documentclass[12pt]{article}
\usepackage{amsmath,amsthm,amscd,amssymb}
\usepackage[mathscr]{eucal}
\usepackage{amssymb}
\usepackage{latexsym}

\newtheorem{Def}{Definition}[section]
\newtheorem{Thm}[Def]{Theorem}
\newtheorem{Prop}[Def]{Proposition}
\newtheorem{Rem}[Def]{Remark}

\newtheorem{Cor}[Def]{Corollary}

\newtheorem{Lem}[Def]{Lemma}
\newtheorem{Conj}[Def]{Conjecture}

\numberwithin{equation}{section}
\newcommand{\Q}{\mathbb{Q}}
\newcommand{\R}{\mathbb{R}}
\newcommand{\C}{\mathbb{C}}
\newcommand{\Z}{\mathbb{Z}}
\newcommand{\F}{\mathbb{F}}
\newcommand{\hh}{\mathbb{H}}

\newcommand{\e}{\mathbf{e}}

\newcommand{\mat}[4]{\begin{pmatrix} #1 & #2 \\ #3 & #4 \end{pmatrix}}
\newcommand{\smat}[4]{\left(\begin{smallmatrix} #1 & #2 \\ #3 & #4 \end{smallmatrix}\right)}

\begin{document}



\title{Structure theorem  for mod $p^m$ singular Siegel modular forms}
\author{Siegfried B\"ocherer and Toshiyuki Kikuta}
\maketitle


\begin{abstract}
We prove that all  mod $p^m$ singular forms of level $N$,
degree $n+r$, and $p$-rank $r$ with $n\ge r$ 
are congruent mod $p^m$ to linear combinations of theta series of degree $r$ attached to quadratic forms of some level. 
Moreover, we prove that,
the levels of
theta series are of the form ``$p\mbox{-power}\times N$''. 
Additionally, in some cases of mod $p$ singular forms with smallest possible weight,
we prove that the levels of theta series should be $p$.
\end{abstract}

\section{Introduction}
Freitag \cite{Frei} showed the following properties concerning singular (Siegel modular) forms.  
\begin{itemize}  \setlength{\itemsep}{-5pt}
\item A Siegel modular form of weight $k$ and degree $n$ is singular if and only if $k<n/2$.
\item The weight $k$ of a singular form with singular rank $r$ should be $k=r/2$.
\item All singular modular forms of singular rank $r$ are
  linear combinations of theta series attached to quadratic forms
  of matrix size $r$. 
\end{itemize} 
We should also mention here that Resnikoff \cite{Res} and much later
Shimura \cite{Shim} considered singular modular forms on more general
tube domains. 

In \cite{Bo-Ki}, the authors defined a notion of mod $p^m$ singular form and proved that 
a congruence holds between the weight $k$ and
the mod $p$ singular rank $r$ (we call it $p$-rank).  
More precisely we proved that $2k-r\equiv 0$ mod $(p-1)p^{m-1}$.
This can be regarded as a mod $p^m$ analogue to Freitag's result
concerning weights.
We also mentioned in the same paper that we can construct some examples
from any true singular modular forms and also from
some Siegel-Eisenstein series of level $1$.
Another type of examples is provided by the work of Nagaoka \cite{Na}
and Katsurada-Nagaoka \cite{Kats-Na}, who studied $p$-adic limits
of Siegel-Eisenstein series and showed in some cases that they are equal
to certain genus theta series of singular rank. 
In particular, some of the Siegel-Eisenstein series studied by them are then congruent to
linear combinations of singular theta series, i.e. true singular forms.
Our conjecture is that, all mod $p^m$ singular forms are
congruent mod $p^m$ to true singular forms with some level, 
i.e., linear combinations of theta series with singular weights
($<\frac{n}{2}$) and some level.

In this paper, we mainly discuss this conjecture and prove it for many cases, in particular for ``strongly'' mod $p^m$ singular forms:
More precisely, we prove the following properties. 
\begin{itemize}  \setlength{\itemsep}{-5pt}
\item All mod $p^m$ singular forms of level $N$, degree $n+r$, and $p$-rank $r$ with $n\ge r$ are congruent mod $p^m$ to 
linear combinations of theta series of degree $r$ attached to quadratic forms of some level (Theorem \ref{Thm:general} (1)). 
\item
We show that the levels of the theta series are of the form
  $p^eN$ with some $e\in {\mathbb N}$
(Theorem \ref{Thm:general} (2)). 
\item In the case where the weights are the smallest possible in
  some sense, then the levels of theta series should be $p$;
  this is proved for the case of mod $p$ singular forms of $p$-rank $2$
  (Theorem \ref{Thm:r=2}).  
\end{itemize}
  
Of course, our results can be regarded as the congruence version of Freitag's results. 
However, there are many new difficulties in mod $p$ cases. 
For example,
if a degree $n$ theta series for a quadratic form $S$ is congruent mod $p$ to a level one modular form,
then this does not specify the level of $S$. This is one of the main
differences to the case over ${\mathbb C}$.
We can use the filtration (weight) of the theta series as a kind of substitute, but this has
rather weak properties.
As new tools we use a strong version of the $q$-expansion principle and
a degree $n$ version of Kitaoka's transformation formula
for theta series. 
These are shown within appendices
of this paper and may be of independent interest.

We remark that we chose to consider only modular forms for groups
of type $\Gamma_0(N)$ and quadratic nebentypus character. 
This allows us to consider congruences modulo rational primes (not modulo prime ideals in
suitable algebraic number fields);
also we wanted to avoid some delicate issues concerning
congruences for Siegel modular forms of nonquadratic nebentypus.
Instead, we show at the end of 
the main text of this paper how
some of our results can be extended to the setting of  mod $p^m$ singular
forms for arbitrary $m$.

This paper is organized as follows.
In Section \ref{Sec:2}, we fix notation and definitions, and review known facts.
In Section  \ref{Sec:3}, we explain our conjectures and main results of this paper. 
In Section \ref{Sec:4}, we provide one of our main tools: An expansion of the rank $r$-part
of the Fourier expansion of any modular form and
its formulation for mod $p^m$ singular forms of $p$-rank $r$.
Then 
in Section \ref{Sec:5} we focus on the mod $p$ case. We first give a
kind of Sturm bond for mod $p$ singular forms and
use it to show that every strongly mod $p$ singular modular form 
is represented by a linear combination of theta series for some quadratic forms. 
In Section \ref{Sec:6}, 
we try to specify the levels of their quadratic forms. 
We discuss also mod $p$ singular forms with $p$-rank $2$ whose weights are the smallest possible.   
In Section \ref{Sec:7}, we extend the results for the mod $p$ case (precisely,
results for general strongly mod $p$ singular forms) 
proved in Section \ref{Sec:5} 
and \ref{Sec:6} to the mod $p^m$ case by induction on $m$. 
Section \ref{Sec:8} and \ref{Sec:9} are appendices: 
We provide two tools that are needed for specification of the levels that 
we do in Section \ref{Sec:6}, namely  
a modified version of the $q$-expansion principle, 
and an extension to degree $n$ of Kitaoka's transformation formula for theta series.

\section{Preliminaries}
\label{Sec:2}
\subsection{Siegel modular forms}
\label{sec:siegel-modular-forms}
Let $\hh_{n}$ be the Siegel upper half space of degree $n$.
We put 
\[{\rm GSp}^+_n(\R):=\{g\in {\rm GL}_{2n}(\R)\;|\; {}^tgJ_ng=n(g)J_n\ \text{for\ some\ } n(g)>0\}, \]
where $J_n:=\smat{0_n}{1_n}{-1_n}{0_n}$. 
We define the action of  ${\rm GSp}^+_n(\R)$ on $\hh_{n}$ by
$gZ = (AZ + B)(CZ + D)^{-1}$ for $Z \in \hh_{n}$, $g \in {\rm GSp}^+_n(\R)$.
For a holomorphic function $F:\mathbb{H}_n\longrightarrow \mathbb{C}$ and a matrix $g=\left( \begin{smallmatrix} A & B \\ C & D \end{smallmatrix}\right)\in {\rm GSp}^+_n(\R)$, 
we define the slash operator in the usual way;
\[(F|_k\; g)(Z):=n(g)^{\frac{nk}{2}}\det (CZ+D)^{-k}F(gZ).\]

Let $N$ be a natural number and $\Gamma _n:={\rm Sp}_n(\mathbb{Z})$ the Siegel modular group (symplectic group with components in $\Z$). 
In this paper, we deal mainly with the congruence subgroup $\Gamma _0^{(n)}(N)$ with level $N$ of $\Gamma _n$ defined as 
\begin{align*}
&\Gamma _0^{(n)}(N):=\left\{ \begin{pmatrix}A & B \\ C & D \end{pmatrix}\in \Gamma _n \: \Big| \: C\equiv 0_n \bmod{N} \right\}.
\end{align*}
We will also use the groups
\begin{align*}
&\Gamma _1^{(n)}(N):=\left\{ \begin{pmatrix}A & B \\ C & D \end{pmatrix}
\in \Gamma _n \: \Big| \: C\equiv 0_n \bmod{N},\
\det A \equiv \det D \equiv 1 \bmod{N} \right\},\\
&\Gamma ^{(n)}(N):=\left\{ \begin{pmatrix}A & B \\ C & D \end{pmatrix}\in \Gamma _n \: \Big{|} \: B\equiv C \equiv 0_n \bmod{N},\ A\equiv D\equiv 1_n \bmod{N} \right\},
\end{align*}
where $\Gamma^{(n)}(N)$ is the so-called principal congruence subgroup of level $N$. We remark that 
\[\Gamma^{(n)}(N)\subset \Gamma _1^{(n)}(N)\subset \Gamma _0^{(n)}(N)\subset \Gamma _n.\] 

For a natural number $k$ and a Dirichlet character
$\chi : (\mathbb{Z}/N\mathbb{Z})^\times \rightarrow \mathbb{C}^\times $, the space
$M_k(\Gamma _0^{(n)}(N),\chi )$
of Siegel modular forms of weight $k$ with
character $\chi$ consists of all of holomorphic
functions $F:\mathbb{H}_n\rightarrow \mathbb{C}$ satisfying
\begin{equation*}
(F|_{k}\: g)(Z)=\chi (\det D)F(Z)\quad \text{for}\quad g=\begin{pmatrix}A & B \\ C & D \end{pmatrix}\in \Gamma_0^{(n)}(N).
\end{equation*}
If $n=1$, the usual condition in the cusps should be added.
When $\chi $ is a trivial character, we write simply $M_k(\Gamma _0^{(n)}(N))$ for $M_k(\Gamma_0^{(n)}(N) ,\chi )$.

Let $\Gamma \supset \Gamma^{(n)}(N)$. 
Similarly as above, we denote by $M_k(\Gamma )$ the space consists of 
all of holomorphic functions $F:\mathbb{H}_n\rightarrow \mathbb{C}$ satisfying
\begin{align*}
(F|_{k}\: g)(Z)=F(Z)\quad \text{for}\quad g=\begin{pmatrix}A & B \\ C & D \end{pmatrix}\in \Gamma.
\end{align*}
In this case also, we have to add the usual condition in the cusps when $n=1$.

For a prime $p$, let $\Gamma $ be $\Gamma _0^{(n)}(p^m)\cap \Gamma ^{(n)}(N) $ or $\Gamma _0^{(n)}(p^m)\cap \Gamma _1^{(n)}(N)$, and $\chi $ a Dirichlet character mod $p$.
In this paper, symbols $M_k(\Gamma ,\chi)$ and $M_k(\Gamma )$ also for such $\Gamma$ are sometimes used, but these are defined in the same way as above.

Note that, for any Dirichlet character $\chi $ mod $N$, we have  
\[M_k(\Gamma ^{(n)}(N))\supset M_k(\Gamma _1^{(n)}(N))\supset M_k(\Gamma _0^{(n)}(N),\chi )
. \]
When $F\in M_k(\Gamma)$ with $\Gamma \supset \Gamma ^{(n)}(N)$, $N$ is called the ``level'' of $F$. Sometimes $\Gamma $ itself is also called the level of $F$.

Any $F \in M_k(\Gamma ^{(n)}(N))$ has a Fourier expansion of the form
\[
F(Z)=\sum_{0\leq T\in \frac{1}{N}\Lambda_n}a_F(T){\boldsymbol e}({\rm tr}(TZ)),
\quad Z\in\mathbb{H}_n,
\]
where ${\boldsymbol e}(x):=e^{2\pi i x}$,  
\[
\Lambda_n
:=\{ T=(t_{ij})\in {\rm Sym}_n(\mathbb{Q})\;|\; t_{ii},\;2t_{ij}\in\mathbb{Z}\; \},
\]
and ${\rm Sym}_n(R)$ is the set of symmetric matrices of size $n$ with components in $R$. 
In particular, if $F \in M_k(\Gamma _1^{(n)}(N))$, the Fourier expansion of $F$ is given in the form 
\[F(Z)=\sum_{0\leq T\in \Lambda_n}a_F(T){\boldsymbol e}({\rm tr}(TZ)). \]

It is known that 
\[M_k(\Gamma _1^{(n)}(N))=\bigoplus _{\chi :(\Z/N\Z)^\times \to \C^\times }M_k(\Gamma_0^{(n)}(N) ,\chi ),\]
where $\chi $ runs over all the Dirichlet characters mod $N$. 
We remark that, if $F\in M_k(\Gamma _0^{(n)}(N),\chi )$, then we have
\[a_F(T[U])=(\det U)^k\chi (\det U)a_F(T)\]
for each $T\in \Lambda _n$ and $U\in {\rm GL}_n(\Z)$. Here we write as $T[U]:={}^t U T U$. 
In particular, if $\chi (-1)=(-1)^k$, we have $a_F(T[U])=a_F(T)$ for 
each $T\in \Lambda _n$ and $U\in {\rm GL}_n(\Z)$. 

Let $\Phi$ be the Siegel $\Phi$-operator defined by  
\[\Phi (F)(Z'):=\lim _{t\to \infty }F\mat{Z'}{0}{0}{it}, \]
where $F\in M_k(\Gamma_0^{(n)}(N) ,\chi )$, $Z'\in \hh _{n-1}$, and $t\in \R$. 
As is well-known, we have $\Phi (F)\in M_k(\Gamma_0^{(n-1)}(N) ,\chi )$ and the Fourier expansion of $\Phi (F)$ is described as
\[\Phi(F)(Z')=\sum _{0\le T\in \Lambda _{n-1}}a_F\mat{T}{0}{0}{0}{\boldsymbol e}({\rm tr}(TZ')). \]
Suppose that $F\in M_k(\Gamma _0^{(n)}(N),\chi )$ satisfies $\Phi (F)\neq 0$ (and then $F\neq 0$). 
Taking $g\in \Gamma _0^{(n)}(N)$ as $g=\smat{-1_n}{0_n}{0_n}{-1_n}$, 
we have 
$F|_k g=(-1)^{nk}F=\chi (-1)^{n}F$
because of the transformation law of $F$. 
Therefore we have $\chi (-1)^n=(-1)^{nk}$. 
On the other hand, by the same property of $\Phi (F)\neq 0$, 
we have $\chi (-1)^{n-1}=(-1)^{(n-1)k}$.  
These imply that $\chi (-1)=(-1)^k$. 
In this case (of $\Phi(F)\neq 0$), we have automatically $a_F(T[U])=a_F(T)$ for 
 each $T\in \Lambda _n$ and $U\in {\rm GL}_n(\Z)$. 

For a subring $R$ of $\mathbb{C}$, let $M_{k}(\Gamma,\chi )_{R}$ (resp. $M_{k}(\Gamma )_{R}$)
denote the $R$-module of all modular forms in $M_{k}(\Gamma, \chi )$ (resp. $M_{k}(\Gamma )$)
 whose Fourier coefficients are in $R$.

\subsection{Congruences for modular forms}
Let $p$ be a prime and $\Z_{(p)}$ the set of $p$-integral rational numbers. 
Let $F_i$ ($i=1$, $2$) be two formal power series of the form
\[F_i=\sum _{T\in \frac{1}{N}\Lambda _{n}}a_{F_i}(T){\boldsymbol e}({\rm tr}(TZ))\]
with $a_{F_i}(T)\in \Z_{(p)}$ for all $T\in \frac{1}{N}\Lambda _n$. 
We write $F_1 \equiv F_2$ mod $p^m$ if $a_{F_1}(T)\equiv a_{F_2}(T)$ mod $p^m$ for all $T \in \frac{1}{N}\Lambda _n$.  

Let $\widetilde{M}_{k}(\Gamma _0^{(n)}(N),\chi )_{p^m}$ be the set of  
$\widetilde{F}=\sum _{T}\widetilde{a_F(T)}{\boldsymbol e}({\rm tr}(TZ))$ with 
$F\in M_{k}(\Gamma _0^{(n)}(N),\chi )_{\Z_{(p)}}$, where $\widetilde{a_F(T)}:=a_F(T)$ mod $p^m$. 
If $m=1$, we write simply $\widetilde{M}_{k}(\Gamma _0^{(n)}(N),\chi )$ for $\widetilde{M}_{k}(\Gamma _0^{(n)}(N),\chi )_{p^m}$. 
Note that, $\widetilde{M}_{k}(\Gamma _0^{(n)}(N),\chi )$ is a vector space over $\F_p$.

We define the filtration weight as 
\begin{align*}
&\omega^{n}_{N,\chi,p^m} (F):=\min \{k \;|\; \widetilde{F} \in \widetilde{M}_k(\Gamma _0^{(n)}(N),\chi )_{p^m} \}. 
\end{align*}
We write also $\omega^{n}_{N,\chi} (F):=\omega^{n}_{N,\chi,p} (F)$.

\begin{Def}
Let $F\in M_k(\Gamma _0^{(n+r)}(N),\chi )_{\Z_{(p)}}$. 
We say that $F$ is ``mod $p^m$ singular'' if 
\begin{itemize}  \setlength{\itemsep}{-5pt}
\item we have $a_F(T)\equiv 0$ mod $p^m$ for any $T\in \Lambda _{n+r}$ with ${\rm rank}(T)>r$, 
\item there exists $T\in \Lambda _{n+r}$ with ${\rm rank}(T)=r$ satisfying $a_F(T)\not \equiv 0$ mod $p$.
\end{itemize} 
We call such $r$ ``$p$-rank'' of $F$.  
Additionally if $n\ge r$, we say that such $F$ is ``strongly mod $p^m$ singular''.
\end{Def}
\begin{Rem} 
\begin{enumerate}  \setlength{\itemsep}{-5pt}
\item
 Such a condition ``strongly
 singular" also plays a crucial role in Freitag's book \cite{Frei}
for the theory over ${\mathbb C}$ for arbitrary congruence subgroups.
\item
If $F$ is nontrivial mod $p^m$ singular, then $F$ satisfies
$\Phi (F)\neq 0$. 
This implies that $\chi (-1)=(-1)^k$ and $a_F(T[U])=a_F(T)$  
for each $T\in \Lambda _n$ and $U\in {\rm GL}_n(\Z)$ (see Page 4). 
In other words, when dealing with nontrivial mod $p^m$ singular modular forms, 
we may assume that $\chi (-1)=(-1)^k$.
\end{enumerate}
\end{Rem}
\begin{Thm}[B\"ocherer-Kikuta \cite{Bo-Ki}]
\label{Thm:Bo-Ki}
Let $n$, $r$, $k$, $N$ be positive integers and $p$ a prime with $p\ge 5$.  
Let $\chi $ be a quadratic Dirichlet character mod $N$ with $\chi (-1)=(-1)^k$. 
Suppose that $F\in M_k(\Gamma ^{(n+r)}_0(N),\chi )_{\Z_{(p)}}$ is mod $p^m$ singular of $p$-rank $r$. 
Then we have $2k-r\equiv 0$ mod $(p-1)p^{m-1}$. 
In particular, $r$ should be even.  
\end{Thm}

It is a classical result by Serre \cite{Se} that a congruence mod $p$ for two
elliptic modular forms $f$ and $g$ for level one implies a congruence of their
weights mod $p-1$. 
We need a version for degree $n$ including levels and quadratic nebentypus:
\begin{Prop}
\label{weightcongruence}
Let $p$ be a prime with $p\ge 5$ and $N$ a positive integer with $p\nmid N$. 
Let  $\chi$ and $\chi'$ be two quadratic Dirichlet characters mod $p$ and
$F\in M_{k}(\Gamma_0^{(n)}(p^m)\cap \Gamma_1^{(n)}(N),\chi)$, $F'\in M_{k'}(
\Gamma_0^{(n)}(p^m) \cap \Gamma_1^{(n)}(N),\chi')$ be two modular forms satisfying $F\equiv F'\bmod p$.
Then $k-k'=t\cdot \frac{p-1}{2}$ holds for some $t\in {\mathbb Z}$ and we have
\[\chi=\chi'\iff t \quad \mbox{even}. \]
\end{Prop}  
\begin{proof} 
We want to apply the results from B\"ocherer-Nagaoka \cite{Bo-Na:3} to get the desired congruences for the weights.
Note that in \cite{Bo-Na:3} only the case of level $\Gamma_1^{(n)}(N)$ with $N$ coprime to $p$ is covered.
 We may apply level change to $F$ and $F'$ to arrive at $G$ and $G'$ of
  level $\Gamma_1^{(n)}(N)$ with $F\equiv G$ mod $p$ and $G'\equiv F'$ mod $p$ with weights $l$ and $l'$.

 If $\chi=\chi'$, we have $k\equiv l$ mod $p-1$ and $k'\equiv l'$ mod $p-1$
  and then, by \cite{Bo-Na:3} $k\equiv k'$ mod $p-1$, i.e. $t$ is even.
  
  If $\chi\not=\chi'$, let us assume that $\chi$ is nontrivial.
  Then $l\equiv k+\frac{p-1}{2}$ mod $p-1$ and the congruence
  $l\equiv l'$ mod $p-1$ implies that $t$ has to be odd. This completes the proof. 
\end{proof}

To formulate our results efficiently, we introduce the following
(somewhat nonstandard). 
\begin{Def}
\label{Eq_Char} 
Suppose that $k-k'=t\cdot \frac{p-1}{2}$ holds for some $t\in {\mathbb Z}$. 
For a prime $p$ and a natural number $N$ coprime to $p$ let $\chi$ and $\chi'$ be two quadratic Dirichlet characters mod
$pN$. 
We write
$$\chi='\chi'$$
if $\chi_N =\chi'_N$, and $\chi_p$ and $\chi'_p$ are related as in
Proposition \ref{weightcongruence}; i.e. 
\[\chi_p=\chi'_p\iff t \quad \mbox{even}. \]
Here $\chi_N$ and $\chi_p$ are the $N$-component and
$p$-component of $\chi$ (and the same for $\chi'$).
In other words, we have
\[\chi =' \chi'\iff \chi =\chi' \left(\frac{*}{p}\right)^t, \]
where $(\frac{*}{p})$ is the unique nontrivial quadratic character mod $p$. 

Note that this notation depends on $k$, $k'$ but it will always be clear form the 
context, which weights are involved.
\end{Def}

\subsection{Theta series for quadratic forms}
Let $m$ be a positive integer. 
For $S$, $T\in \Lambda _m$, we write $S\sim T$ mod ${\rm GL}_m(\Z)$ if 
there exists $U\in {\rm GL}_m(\Z)$ such that $S[U]=T$.
Here we put $S[U]:={}^tUSU$. 
We say that $S$ and $T$ are ``${\rm GL}_m(\Z)$-equivalent'' if $S\sim T$ mod ${\rm GL}_m(\Z)$. 
We denote by $\Lambda _m^+$ the set of all positive definite elements of $\Lambda _m$. 
We put $L:=\Lambda _m$ or $\Lambda ^+_m$. 
We write $L/{\rm GL}_m(\Z)$ for $L/\sim $  
the set of representatives of ${\rm GL}_m(\Z)$-inequivalence classes in $L$.   

Let $m$ be even. 
For $S\in \Lambda _m^+$, we define the theta series of degree $n$ in the usual way:
\[\theta _S^{(n)}(Z):=\sum _{X\in \Z^{m,n}}{\boldsymbol e}({\rm tr}(S[X]Z))\quad (Z\in \hh_{n}), \] 
where $\Z^{m,n}$ is the set of  $m\times n$ matrices with integral components and $S[X]:={}^tXSX$ and we write 
${\boldsymbol e}(x):=e^{2\pi i x}$. 
We define the level of $S$ as 
\[{\rm level}(S):=\min\{N\in \Z_{\ge 1} \;|\; N(2S)^{-1}\in 2\Lambda _m\}. \]
Then $\theta _S^{(n)}$ defines an element of $M_{\frac{m}{2}}(\Gamma _0^{(n)}(N),\chi _S)$, 
where $N={\rm level}(S)$, $\chi _S$ is a Dirichlet character mod $N$ defined by 
\[\chi _S(d)={\rm sign} (d)^\frac{m}{2} \left( \frac{(-1)^\frac{m}{2}\det 2S}{|d|} \right).\] 
We denote by ${\rm cont}(S)$ the content of $S$ defined as
\[{\rm cont}(S):=\max\{C\in \Z_{\ge 1}\;|\; C^{-1}S\in \Lambda _n \}. \] 

For fixed $S\in \Lambda^+ _{m}$ and $T\in \Lambda _n$, we put 
\[A(S,T):=\sharp\{X \in \Z^{m,n} \;|\; S[X]=T \}.\]
Using this notation, we can write the Fourier expansion of the theta series in the
form 
\[\theta _S^{(n)}(Z)=\sum _{T\in \Lambda _n}A(S,T){\boldsymbol e}({\rm tr}(TZ)). \]

In the above definitions, 
$S$ and $T$ are restricted to symmetric half integral matrices, 
but the same symbols (such as $\theta _S^{(n)}(Z)$ and $A(S,T)$) are used for symmetric matrices with rational components.
Then $A(S,S)$ is the order of automorphism group of $S$;
\[A(S,S)=\sharp \{U\in {\rm GL}_m(\Z)\;|\; S[U]=S\}. \]
By looking at minimal polynomials, one sees that ${\rm GL}_m({\mathbb Z})$ cannot
contain elements of order $p$ if $p>m+1$. 
From this we obtain the very useful statement that
\begin{align}
\label{Mink}
A(S,S)\not \equiv 0 \bmod{p} \quad \text{if}\quad p> m+1
\end{align}
for any rational positive definite symmetric matrix $S$ of size $m$.

For later use, we introduce some results on the representation
numbers for binary quadratic forms: 
\begin{Thm}[Dirichlet, Weber (see Kani \cite{Kani}, Lemma 8, page 4)]
\label{Thm:Diri}
Let $T$, $T_i\in \Lambda _2$ ($i=1,2$) be primitive forms, i.e. ${\rm cont}(T)={\rm cont}(T_i)=1$. 
Assume that $T_1$ and $T_2$ have a same discriminant $D<0$. Then we have the following statements. 
\begin{enumerate}  \setlength{\itemsep}{-5pt}
\item
There are infinitely many primes $l$ such that $A(T,l)>0$.
\item
If there exists a prime $l$ with $l\nmid D$ such that $A(T_1,l)>0$ and $A(T_2,l)>0$, then $T_1\sim T_2$ mod ${\rm GL}_2(\Z)$.
\end{enumerate}
\end{Thm}

\begin{Thm}[B\"ocherer-Nagaoka \cite{Bo-Na:2}]
\label{Thm:Bo-Na}
Let $S\in \Lambda _{2}^+$ be of level $p$. Let $n$ be a positive integer. 
Then there exists $F\in M_{\frac{p+1}{2}}(\Gamma _n)_{\Z_{(p)}}$ such that $F\equiv \theta ^{(n)}_S$ mod $p$. 
\end{Thm}
For a quadratic form $S\in \Lambda _m^+$, we put $\omega ^{n}_{N,\chi ,p^m}(S):=\omega ^{n}_{N,\chi,p^m}(\theta ^{(n)}_S)$,
$\omega ^{n}_{N,\chi }(S):=\omega ^{n}_{N,\chi,p}(\theta ^{(n)}_S)$.

\section{Conjectures and Results}
\label{Sec:3}
We begin by stating our conjecture in the most general situation which does not specify the weight and level.  
\begin{Conj}
\label{Conj0}
Any mod $p^m$ singular form is congruent mod $p^m$ to some true singular form. 
\end{Conj}
Our first result states that this conjecture is true in the strongly mod $p^m$ singular case (if $p$ is not small).
\begin{Thm}
\label{Thm:general}
Let $n$, $k$, $N$ be positive integers, $r$ an even integer with $n\ge r$. 
Let $p$ be a prime with $p>r+1$ and $\chi $ a quadratic
Dirichlet character mod $N$ with $\chi (-1)=(-1)^k$.  
Suppose that $F\in M_{k}(\Gamma _0^{(n+r)}(N),\chi )_{\Z_{(p)}}$ is
mod $p^m$ singular of $p$-rank $r$. 
Then we  have the following statements. 
\begin{enumerate}  \setlength{\itemsep}{-5pt}
\item
There are finitely many $S\in \Lambda _r^+$ such that 
\[F\equiv \sum_S c_S\theta_S^{(n+r)} \bmod{p^m} \quad (c_S\in \Z_{(p)}) \]
and $\widetilde{\theta _{S}^{(n)}}\in \widetilde{M}_k(\Gamma ^{(n)}_0(N),\chi )_{p^{m-\nu}}$ (and hence $\omega ^n_{N,\chi ,p^{m-\nu }}(S)\le k$). 
Here $\nu :=\nu_p(c_S)$ and $\nu _p$ is the additive valuation on $\Q$ normalized so that $\nu_p(p)=1$.  
Moreover, all $S$ involved satisfy $\chi='\chi_S$. 
\item
For a suitable $e\in {\mathbb N}$, all of $S\in \Lambda _r^+$ appearing in (1) satisfy that ${\rm level}(S)\mid p^eN$. 
\end{enumerate}
\end{Thm}
\begin{Rem}
\begin{enumerate}  \setlength{\itemsep}{-5pt}
\item
Actually, each $c_S$ is described in terms of the primitive Fourier coefficient for $S$ of $F$. 
For details, see the proof in Section \ref{Sec:7}. 
\item
The statement on $\widetilde{\theta _{S}^{(n)}}\in \widetilde{M}_k(\Gamma ^{(n)}_0(N),\chi )_{p^{m-\nu}}$ can be rephrased by $c_S\theta_S^{(n)}\in \widetilde{M}_k(\Gamma ^{(n)}_0(N),\chi )_{p^{m}}$.
Note also that $\nu =0$ if $m=1$.  
\item
We emphasize that we do not know anything about $e$.
\end{enumerate}
\end{Rem}
We expect that the theorem above holds in the most general case:
\begin{Conj} 
Theorem \ref{Thm:general} should hold for any prime (not only for $p>r+1$)
and without the assumption $n\geq r$. 
\end{Conj}

The smallest weights (see Theorem \ref{Thm:Bo-Ki}) where we can expect mod $p$ singular forms,
which are not true singular forms, is of special interest. 
In these cases, we expect also the power $e$ of $p$ in the levels of theta series to be the smallest possible:   
\begin{Conj}
\label{Conj4}
Let $n$ be a positive integer, $r$ an even integer, and $p$ a prime with $p\ge n+r+3$. 
We put 
\[k=k(p,r):=
\begin{cases} 
r/2+(p-1)/2\quad &\text{if}\quad  r\equiv 2 \bmod{4},\ p\equiv -1 \bmod{4} \\
r/2+p-1\quad &\text{if}\quad  r\equiv 0 \bmod{4}. 
\end{cases}\]
Let $N$ be a positive integer and $\chi $ a quadratic Dirichlet character mod $N$ with $\chi (-1)=(-1)^k$. 
Suppose that $F\in M_{k}(\Gamma _0^{(n+r)}(N),\chi )_{\Z_{(p)}}$ is mod $p$ singular of $p$-rank $r$.  
Then we have 
\begin{align*}
F\equiv \sum _{\substack{S \in \Lambda _{r} / {\rm GL}_r(\Z) \\ {\rm level}(S)\mid pN}} c_S\theta ^{(n+r)}_{S} \bmod{p}\quad (c_S\in \mathbb{Z}_{(p)}). 
\end{align*} 
\end{Conj}
In the special case of $N=1$, $r=2$, we can prove this conjecture: 
\begin{Thm}
\label{Thm:r=2} 
Let $n$ be a positive integer and $p$ a prime with $p\ge n+5$. 
Suppose that $F\in M_{\frac{p+1}{2}}(\Gamma _{n+2})_{\Z_{(p)}}$ is mod $p$ singular of $p$-rank $2$.   
Then we have 
\begin{align*}
F\equiv \sum _{\substack{S \in \Lambda _{2} / {\rm GL}_2(\Z) \\ {\rm level}(S)=p}} c_S\theta ^{(n+2)}_{S} \bmod{p}\quad (c_S\in \mathbb{Z}_{(p)}). 
\end{align*}
Moreover for any such $F$ the degree one form $f:=\Phi^{n+1}(F)$
is nonzero mod $p$.
\end{Thm}

\begin{Rem}
We remark that Theorem \ref{Thm:r=2}
has some potential for showing the nonvanishing of
Fourier coefficients mod $p$ in some interesting cases, e.g. 
a degree $3$ modular form satisfying $\Phi^2(F)\equiv 0$ mod $p$
cannot be mod $p$ singular. 
This implies the existence of some nonvanishing
rank 3-Fourier coefficients $a_F(T)\bmod p$ of Klingen-Eisenstein series
$F:=E^{3,2}_k$ attached to a cusp form $h$ of degree 2, provided that its
Fourier coefficients are in ${\mathbb Z}_{(p)}$.
Note that such Fourier coefficients are quite delicate,
given by some critical values of
$L$-series attached to $h$ and to $T\in \Lambda^+_3$, if $h$ is a Hecke eigenform
(see \cite{Bo,Kli} for more details on Klingen-Eisenstein series).
To cover more general
cases, versions of this for congruences modulo  prime ideals
in the field generated by the Hecke eigenvalues of $h$ 
will be necessary.
\end{Rem}

\section{Refinement of Freitag's expansion}
\label{Sec:4}
In this section, inspired by Freitag \cite{FreiA,Frei}, we give a
formal expansion of the ``singular part'' of the Fourier expansion of 
any modular form and apply it to mod $p^m$ singular forms.

We fix some notation. 
Let $M_n(R)$ be the set of all $n\times n$ matrices whose components are in $R$.  
We put $M^{(r)}_n(\Z):=\{M\in M_n(\Z)\;|\; {\rm rank}(M)=r\}$ and $M^{*}_n(\Z):=M^{(n)}_n(\Z)$. Similarly we write  
$\Lambda _n^{(r)}:=\{T\in \Lambda _n\;|\; {\rm rank}(T)=r\}$ ($\Lambda _n^+=\Lambda _n^{(n)}$).
Let $F$ be a modular form of degree $n+r$ with Fourier expansion 
\[F(Z)=\sum _{T\in \Lambda _{n+r}}a(T){\boldsymbol e}({\rm tr}(TZ)).\]
We write $a(S):=a\smat{0}{0}{0}{S}$ when $S\in \Lambda _r$.
We define a subseries $F_{[r]}$ of $F$ as
\[F_{[r]}(Z):=\sum _{T\in \Lambda _{n+r}^{(r)}} a(T){\boldsymbol e}({\rm tr}(TZ)). \]

In B\"ocherer-Raghavan \cite{Bo-Ra} (see page 82 and 83),  the notion of ``primitive Fourier coefficient'' was introduced; 
we denote it by $a^*(S)$ for $S$ positive definite. Namely, 
$a^*(S)$ is defined by the formula
\[a(S)=\sum _{\substack{G\in {\rm GL}_r(\Z)\backslash M^*_r(\Z)\\S[G^{-1}]\in \Lambda _r}}a^*(S[G^{-1}]). \] 
We recall that by this formula, we can define a new ${\rm GL}_r(\Z)$-invariant function $a^*(S)$
starting from the ${\rm GL}_r(\Z)$-invariant function
$T\longmapsto a(T)$ on $\Lambda_r$.

As explained in \cite{Bo-Ra}, this can also be written as
\begin{align}
\label{Eq:Pri}
a(T)=\sum _{S\in \Lambda _r^+ / {\rm GL}_r(\Z)} \frac{1}{\epsilon (S)}\sum _{\substack{ W\in M_r^{*}(\Z) \\ S[W]=T}} a^*(S),  
\end{align}
where $\epsilon (S):=A(S,S)$. We use this version.

Using the Fourier coefficients $a(S)$ with $S\in \Lambda _r^+$ and their modification, 
a slight refinement of Freitag's argument
gives by formal rearrangement of the Fourier expansion our
crucial identity. Note that the statements are only claiming the equality
of the Fourier coefficients on both sides (ignoring questions concerning
convergence!).

\begin{Lem}
\label{Lem:Refine}
Let $\chi $ be a Dirichlet character mod $N$ such that $\chi (-1)=(-1)^k$. 
Let $F\in M_k(\Gamma ^{(n+r)}_0(N),\chi )$. 
Then we have
\begin{align}
\label{Eq:0.1}
F_{[r]}(Z)&=\sum _{S\in \Lambda_{r}^{+}/{\rm GL_r}(\Z) }
\frac{a^*(S)}{\epsilon(S)}\sum _{\substack{(X_1,X_2)\in \Z^{r,r}\times \Z^{r,n}\\{\rm rank}(X_1,X_2)=r}}{\boldsymbol e}({\rm tr}(S[(X_1,X_2)]Z)). 
\end{align}
\end{Lem}
Note that
\begin{align*}
\sum _{\substack{(X_1,X_2)\in \Z^{r,r}\times \Z^{r,n}}}{\boldsymbol e}({\rm tr}(S[(X_1,X_2)]Z))&=\sum _{\substack{X\in \Z^{r,n+r}}}{\boldsymbol e}({\rm tr}(S[X]Z))=\theta _S^{(n+r)}(Z).
\end{align*}
Hence if we can prove this lemma, then $F_{[r]}$ can be expressed by a infinite linear combination of (subseries of) theta series;
\begin{align}
\label{Eq:0.15}
F_{[r]}=\sum _{S\in  \Lambda_{r}^{+}/{\rm GL_r}(\Z)}\frac{a^*(S)}{\epsilon(S)}(\theta ^{(n+r)}_S)_{[r]}.
\end{align}

\begin{proof}[Proof of Lemma \ref{Lem:Refine}]
For $X_1\in \Z^{r,r}$, $ X_2\in \Z^{r,n}$, we can take $W\in M^*_r(\Z)$ such that $(X_1,X_2)=W(G_1,G_2)$ and $\smat{*}{*}{G_1}{G_2}\in {\rm GL}_{n+r}(\Z)$.
This $W$ can be regarded as ``gcd'' of $X_1$ and $X_2$. 
We observe that such $W$ is unique up to a factor in ${\rm GL}_r(\Z)$ from the right. 
We switch this action of ${\rm GL}_r(\Z)$ to $(G_1,G_2)$. 
Then we can write the right hand side of (\ref{Eq:0.1}) as
\begin{align}
\label{Eq:0.2}
\sum _{S\in  \Lambda_{r}^{+}/{\rm GL_r}(\Z)}\frac{a^*(S)}{\epsilon(S)}\sum _{W\in M^*_r(\Z)} \sum _{\substack{(G_1,G_2)\in 
{\rm GL}_r(\Z)\backslash \Z^{r,r}\times \Z^{r,n}\\ \smat{*}{*}{G_1}{G_2}\in {\rm GL}_{n+r}(\Z)
}}
{\boldsymbol e}({\rm tr}(S[W][(G_1,G_2)]Z))
\end{align}
where ${\rm GL}_r(\Z)\backslash \Z^{r,r}\times \Z^{r,n}=\sim \backslash \Z^{r,r}\times \Z^{r,n}$ and $(X_1, X_2)\sim (Y_1,Y_2)$ means that $(X_1, X_2)=G(Y_1,Y_2)$ for some $G\in {\rm GL}_r(\Z)$.

We put $S[W]=T$ and we rewrite the summation over $S$ as over $T$.
Then (\ref{Eq:0.2}) becomes 
\begin{align}
\label{Eq:0.3}
\sum _{T\in \Lambda _{r}^{+}}a(T) \sum _{
\substack{(G_1,G_2)\in 
{\rm GL}_r(\Z)\backslash \Z^{r,r}\times \Z^{r,n}\\ \smat{*}{*}{G_1}{G_2}\in {\rm GL}_{n+r}(\Z)
}}
{\boldsymbol e}({\rm tr}(T[(G_1,G_2)]Z))
\end{align}   
because of (\ref{Eq:Pri}).

If we put $U:=\smat{*}{*}{G_1}{G_2}$, then we have $T[(G_1,G_2)]=\smat{0}{0}{0}T[U]$. 
Therefore we have $a(T)=a\smat{0}{0}{0}{T}=a(T[(G_1,G_2)])$. 
Then (\ref{Eq:0.3}) can be written as 
\begin{align*}
\sum _{T\in \Lambda _{r}^{+}}&\sum _{
\substack{(G_1,G_2)\in 
{\rm GL}_r(\Z)\backslash \Z^{r,r}\times \Z^{r,n}\\ \smat{*}{*}{G_1}{G_2}\in {\rm GL}_{n+r}(\Z)
}}
a(T[(G_1,G_2)]) {\boldsymbol e}({\rm tr}(T[(G_1,G_2)]Z))\\
&= \sum _{T\in \Lambda _{n+r}^{(r)}}a(T){\boldsymbol e}({\rm tr}(TZ))=F_{[r]}(Z). 
\end{align*}

Here the first equality in this formula follows from the fact that, if $T\in \Lambda _r^+$ and $(G_1,G_2)$ run as in the subscript, then $T[(G_1,G_2)]$ runs over all elements of $\Lambda _{n+r}^{(r)}$.
\end{proof}
Note that $F_{[r]}$ is not a modular form because some part of the Fourier expansion is missing. 

Let $Z_1\in \hh_r$, $Z_2\in \hh_n$. 
Consider the restriction of $F$ to $Z=\smat{Z_1}{0}{0}{Z_2}$; 
\[F\mat{Z_1}{0}{0}{Z_2}=\sum _{\substack{\smat{T}{*}{*}{*}\in \Lambda _{n+r} \\ T \in \Lambda _{r}}} \phi _T (Z_2)\e ({\rm tr}(TZ_1)). \]
Then we have $\phi _T(Z_2)\in M_k(\Gamma ^{(n)}_0(N),\chi )$ for any $T\in \Lambda _{r}$. For the proof of this fact, we refer to Andrianov \cite{And} (page 83 and 84). 

On the other hand, we denote by $F^{\sharp}$ the subseries of $F=\sum _{\mathfrak{T}\in \Lambda _{n+r}}a_F(\mathfrak{T})\e ({\rm tr}(\mathfrak{T}Z))$, characterized by 
\[\mathfrak{T}=\mat{T}{*}{*}{*}\quad \text{with} \quad T \in \Lambda _r^+. \] 
Namely we put 
\[F^{\sharp}(Z):=\sum _{\substack{\mathfrak{T}=\smat{T}{*}{*}{*}\in \Lambda _{n+r} \\ T \in \Lambda _r^+}}a_F(\mathfrak{T})\e ({\rm tr}(\mathfrak{T}Z)). \]
Then we have
\[F^{\sharp}\mat{Z_1}{0}{0}{Z_2}=\sum _{\substack{
\smat{T}{*}{*}{*}\in \Lambda _{n+r} \\ T \in \Lambda _r^+}} \phi _T (Z_2)\e ({\rm tr}(TZ_1)).\]
Note that still we have $\phi _T (Z_2)\in M_k(\Gamma ^{(n)}_0(N),\chi )$ for any $T\in \Lambda _r^+$. 

Now assume that $F$ is mod $p^m$ singular of $p$-rank $r$. 
Then from $F$ being  mod $p^m$ singular we obtain 
$F^{\sharp} \smat{Z_1}{0}{0}{Z_2}\equiv (F_{[r]})^{\sharp}\smat{Z_1}{0}{0}{Z_2}$ mod $p^m$. 
By Lemma \ref{Lem:Refine}, we have  
\begin{align*}
  (F_{[r]})^{\sharp}\mat{Z_1}{0}{0}{Z_2}&\equiv
  \left( \sum _{S\in \Lambda_{r}^{+}/{\rm GL_r}(\Z) }\frac{a^*(S)}{\epsilon(S)}
  \sum _{\substack{(X_1,X_2)\in \Z^{r,r}\times \Z^{r,n}\\{\rm rank}(X_1,X_2)=r}}
       {\boldsymbol e}({\rm tr}(S[X_1]Z_1)
       {\boldsymbol e}({\rm tr}(S[X_2]Z_2))\right)^{\sharp}\\
       &\equiv\sum _{S\in \Lambda_{r}^{+}/{\rm GL_r}(\Z) }\frac{a^*(S)}
       {\epsilon(S)}\sum _{\substack{(X_1,X_2)\in
           \Z^{r,r}\times \Z^{r,n}\\{\rm rank}(X_1)=r}}
       {\boldsymbol e}({\rm tr}(S[X_1]Z_1)
       {\boldsymbol e}({\rm tr}(S[X_2]Z_2))\\
       &\equiv \sum _{S\in \Lambda_{r}^{+}/{\rm GL_r}(\Z) }\frac{a^*(S)}
       {\epsilon(S)}\sum _{\substack{X_1\in \Z^{r,r}\\{\rm rank}(X_1)=r}}
       {\boldsymbol e}({\rm tr}(S[X_1]Z_1))\sum _{\substack{X_2\in \Z^{r,n}}}
       {\boldsymbol e}({\rm tr}(S[X_2]Z_2))\\
       &\equiv
       \sum _{S\in \Lambda_{r}^{+}/{\rm GL_r}(\Z) }
       \frac{a^*(S)}{\epsilon(S)}(\theta _S^{(r)}(Z_1))_{[r]}
       \theta _S^{(n)}(Z_2)\quad \bmod p^m. 
\end{align*} 
This implies that
\[\phi _T(Z_2)\equiv \sum _{S\in \Lambda_{r}^{+}/{\rm GL_r}(\Z) }
A(S,T)\frac{a^*(S)}{\epsilon(S)}\theta _S^{(n)}(Z_2) \bmod{p^m}  \]
for any $T\in \Lambda _r^+$. 
Hence we obtain that 
\[\sum _{S\in \Lambda_{r}^{+}/{\rm GL_r}(\Z) }A(S,T)\frac{a^*(S)}{\epsilon(S)}\theta _S^{(n)}(Z_2) \bmod{p^m} \in \widetilde{M}_k(\Gamma ^{(n)}_0(N),\chi )_{p^m}. \]

In the ordinary case (over $\C$), the key in Freitag's setting would be
that $a(S)$ can be different from zero only if ${\rm level}(S)\mid N$. 
In our mod $p^m$ setting, we get a condition on the filtration
of $\theta _S^{(n)}$ for $S\in \Lambda _r^+$ with $a^*(S)\not \equiv 0$ mod $p^m$. More precisely, we get the following property.  
\begin{Prop}
\label{Prop:filt}
Let $n$, $k$, $N$ be positive integers and $r$ an even integer. 
Let $p$ be a prime with $p> r+1$ and $\chi $ a quadratic Dirichlet character mod $N$ with $\chi (-1)=(-1)^k$.  
Suppose that $F\in M_k(\Gamma _0^{(n+r)}(N),\chi )_{\Z_{(p)}}$ is mod $p^m$ singular of $p$-rank $r$. 
Then we have $a^*(S)\theta _S^{(n)}$ mod $p^m$ $\in \widetilde{M}_k(\Gamma ^{(n)}_0(N),\chi )_{p^m}$ for any  $S\in \Lambda _r^+$. 
In particular, if $S\in \Lambda _{r}^+$ satisfies $a^*(S)\not \equiv 0$ mod $p^m$,  
then we have $\theta _S^{(n)}$ mod $p^m$ $\in \widetilde{M}_k(\Gamma ^{(n)}_0(N),\chi )_{p^{m-\nu}}$ (and hence  
$\omega _{N,\chi,p^{m-\nu }}^n(S)\le k$) with $\nu :=\nu _p(a^*(S))$.  
Moreover $\chi ='\chi _S$ holds. 
\end{Prop}

\begin{proof}
Seeking a contradiction, we suppose that there exists $S$ such that the claim is not true. 
Let $S_0$ be one of $S$ such that $\det S_0$ is minimal among such $S$. 
Then we consider 
\begin{align*}
\phi _{S_0}-&\sum _{\substack{S\in \Lambda_{r}^{+}/{\rm GL_r}(\Z) \\ \det S<\det S_0}}A(S,S_0)\frac{a^*(S)}{\epsilon(S)}\theta _S^{(n)}\\
&\equiv\sum  _{\substack{S\in \Lambda_{r}^{+}/{\rm GL_r}(\Z) \\ \det S\ge \det S_0}}A(S,S_0)\frac{a^*(S)}{\epsilon(S)}\theta _S^{(n)}\\
&\equiv A(S_0,S_0)\frac{a^*(S_0)}{\epsilon(S_0)}\theta _{S_0}^{(n)}\bmod p^m. 
\end{align*}
Here the last congruence follows from the facts that
$\det S>\det S_0$ implies $A(S,S_0)=0$ and $\det S=\det S_0$,
$S\not \sim S_0$ mod ${\rm GL}_r(\Z)$ implies $A(S,S_0)=0$.
By the assumption, we have
\[\phi _{S_0}-\sum _{\substack{S\in {\rm GL}_r(\Z)\backslash \Lambda _r ^+ \\ \det S<\det S_0}}A(S,S_0)\frac{a^*(S)}{\epsilon(S)}\theta _S^{(n)}\bmod{p^m} \in
\widetilde{M}_k(\Gamma ^{(n)}_0(N),\chi )_{p^m}. \] 
This implies $A(S_0,S_0)\frac{a^*(S_0)}{\epsilon(S_0)}
\theta _{S_0}^{(n)}\bmod{p^m} \in \widetilde{M}_k(\Gamma ^{(n)}_0(N),\chi )_{p^m}$. 
By the fact (\ref{Mink}), we have $p\nmid A(S_0,S_0)$.
This shows $a^*(S_0)\theta _{S_0}^{(n)}\bmod p^m\in
  \widetilde{M}_k(\Gamma ^{(n)}_0(N),\chi )_{p^m}$. 
This is a contradiction.  
Hence we have $a^*(S)\theta _S^{(n)} \bmod p^m\in
  \widetilde{M}_k(\Gamma ^{(n)}_0(N),\chi )_{p^m}$ for any
  $S\in \Lambda _{r}^+$. 
  In particular if $a^*(S)\not \equiv 0$ mod $p^m$, then we have $\theta _S^{(n)}$ mod $p^m$ $\in \widetilde{M}_k(\Gamma ^{(n)}_0(N),\chi )_{p^{m-\nu}}$ with $\nu :=\nu _p(a^*(S))$, and therefore  
$\omega _{N,\chi,p^{m-\nu }}^n(S)\le k$. 
The statement $\chi ='\chi _S$ follows from Proposition \ref{weightcongruence}. 
\end{proof}
For later use, we mention a simple consequence of Proposition \ref{Prop:filt} for the mod $p$ case. 
\begin{Cor}
\label{Cor:filt}
Let $n$, $k$, $N$ be positive integers, $r$ an even integer with $n\ge r$. 
Let $p$ be a prime with $p> r+1$ and
$\chi $ a quadratic Dirichlet character mod $N$ with $\chi (-1)=(-1)^k$.  
Suppose that $F\in  M_k(\Gamma ^{(n+r)}_0(N),\chi )_{\Z_{(p)}}$ is mod $p$ singular of $p$-rank $r$.  
\begin{enumerate}  \setlength{\itemsep}{-5pt}
\item
For any $S\in \Lambda _{r}^+$ with $a^*(S)\not \equiv 0$ mod $p$, we have $\widetilde{\theta _S^{(r)}} \in \widetilde{M}_k(\Gamma ^{(r)}_0(N),\chi )$. 
\item
For $S\in \Lambda _{r}^+$ with $a(S)\not \equiv 0$ mod $p$ such that $\det S$ is minimal in such $S$, 
we have $\widetilde{\theta _S^{(n)}} \in \widetilde{M}_k(\Gamma ^{(n)}_0(N),\chi )$ and therefore $\omega _{N,\chi }^n(S)\le k$.
\end{enumerate}
\end{Cor}
\begin{Rem}
  The statement (2) can be proved also by the mod $p$ version of
  Freitag's original arguments in \cite{Frei}. 
Our strategy can be viewed as a refinement of his method.
\end{Rem}
\begin{proof}[Proof of Corollary \ref{Cor:filt}]
(1) By Proposition \ref{Prop:filt}, we have $\widetilde{\theta _S^{(n)}} \in \widetilde{M}_k(\Gamma ^{(n)}_0(N),\chi )$. 
Hence the claim follows from $\Phi ^{n-r}(\theta _S^{(n)})=\theta _S^{(r)}$ immediately. \\
(2) If $S\in \Lambda _{r}^+$ satisfies such the minimality condition, we have $a^*(S)=a(S)$ by the definition of $a^*(S)$. 
The claim follows from this fact.  
\end{proof}

\section{Description by theta series for mod $p$ case}
\label{Sec:5}
In this section, we prove Theorem \ref{Thm:general} (1) for $m=1$. 
Therefore, we treat with the case of mod $p$ in this section. 
The goal is to show that all  strongly mod $p$ singular modular forms
are represented by a linear combination of finitely many  theta series.
In our method, two new tools are  important:
the existence of an ``abstract Sturm bound"
for detecting mod $p$ singular forms
and the  refinement of Freitag's expansion as exposed in the previous section.
The nice thing is that we do not need to consider the
exact level of the theta series.

\subsection{Abstract Sturm bounds}
In general, ``Sturm bounds" give an
explicit finite set of $T\in \Lambda_n$ such that a modular 
form of degree $n$ with Fourier coefficients in ${\mathbb Z}_{(p)}$ must be congruent
 mod $p$ to zero if 
  the Fourier coefficients for all $T$ in that finite set are divisible by $p$,
see e.g. \cite{RR, Stu}. 
Such an explicit finite set would usually contain quadratic forms of all ranks.
We need a version, which involves only quadratic forms of
rank $>r$ (finitely many) to detect mod $p$ singular forms.
We only need the existence of
such a set, but we do not discuss explicit bounds for it.
 
 Let $\chi $ be a quadratic Dirichlet character mod $N$. 
Let $M_{k,r}^{p\text{-sing}}(\Gamma _0^{(n+r)}(N),\chi)$ be the submodule of
$M_k(\Gamma _0^{(n+r)}(N),\chi )_{\Z_{(p)}}$
consisting of all mod $p$ singular modular forms with $p$-rank $\le r$. 
We denote by $\widetilde{M}_{k,r}^{p\text{-sing}}(\Gamma _0^{(n+r)}(N),\chi)$ the set of  reduction mod $p$ of elements of $M_{k,r}^{p\text{-sing}}(\Gamma _0^{(n+r)}(N),\chi)$.  
This is a subspace over $\F_p$ of the vector space $\widetilde{M}_k(\Gamma _0^{(n+r)}(N),\chi )$. 
We consider the quotient space 
\[V=V_{k,r}:=\widetilde{M}_k(\Gamma _0^{(n+r)}(N),\chi )/\widetilde{M}_{k,r}^{p\text{-sing}}(\Gamma _0^{(n+r)}(N),\chi ).\] 

Then we have $\dim V<\infty$.
For fixed $T\in \Lambda_{n+r}$ with $\mbox{rank}(T)>r$ we define a linear map
$\ell_T:V\longrightarrow {\mathbb F}_p$ by
$$\ell_T(\widetilde{F}+ \widetilde{M}_{k,r}^{p\text{-sing}}(\Gamma _0^{(n+r)}(N),\chi )):=\widetilde{a_F(T)}.$$
Clearly, the set $L$ of all such $\ell_T$ is ``total'' for $V$, i.e.
the intersection of the kernels of all 
$\ell_T$ is trivial.
By linear algebra ($V$ is of finite dimension!)
we can choose a finite subset
${\mathcal T}_{n+r,r}
=\{\ell_{T_1},\dots ,
\ell_{T_d}\}$ of $L$
which is still total for $V$.
\\[0.3cm]
{\bf Conclusion:} {\it In the situation above, there exist finitely many
$T_1,\dots ,T_d$  with all $T_j$ of rank larger than r such that
  for all $F\in M_k(\Gamma^{(n)}_0(N),\chi )_{{\mathbb Z}_{(p)}}$ the vanishing of
  $\widetilde{a_F(T_j)}$
for all $T_j$ implies that $F$ is mod $p$ singular of $p$-rank $\leq r$. }
\\[0.3cm]
We call such $\mathcal T_{n+r,r}$ a ``Sturm set'' for mod $p$ singular forms.

\subsection{Proof of Theorem \ref{Thm:general} (1) for $m=1$}
\label{Subsec:4.3}
For the Sturm set ${\mathcal T}_{r,r-1}$ $(\subset \Lambda_r^+)$ corresponding to $M_{k,r-1}^{p\text{-sing}}(\Gamma _0^{(r)}(N))$,
we take a natural number $M$ such that 
\[M>\max\{\det T\;|\;T\in {\mathcal T}_{r,r-1}\}.\] 
We put \[G:=F-\sum_{\substack{S\in \Lambda_r^+/{\rm GL}_r({\mathbb Z})\\ \det S<M}}\frac{a^*(S)}{\epsilon (S)}\theta_S^{(n+r)}\]
and consider $g:=\Phi ^{n}(G)$. 
Note here that the summation over $S$ is finite. 
Then we have 
\[g=\Phi^{n}(F)-\sum_{\substack{S\in \Lambda_r^+/{\rm GL}_r({\mathbb Z})\\ \det S<M}} \frac{a^*(S)}{\epsilon (S)}\theta_S^{(r)}\]
and $\widetilde{g}\in \widetilde{M}_{k}(\Gamma _0^{(r)}(N),\chi)$ by Corollary \ref{Cor:filt} (1). 

We prove now that $\widetilde{g}\in \widetilde{M}_{k,{r-1}}^{p\text{-sing}}(\Gamma _0^{(r)}(N),\chi )$. 
It suffices to check the Fourier coefficients $a_g(T)$ for all $T\in {\mathcal T}_{r,r-1}$. 
It follows from (\ref{Eq:0.15}) that 
\[a_g(T)=\sum _{\substack{S\in \Lambda_r^+ /{\rm GL}_r({\mathbb Z})\\
    \det S\ge M}}\frac{a^*(S)}{\epsilon (S)}A(S,T) \]
for any $T\in \Lambda _r^+$.
If $S[X]=T$ with $X\in \Z^{r,r}$ then $\det S[X]=\det S (\det X)^2 =\det T$ and hence $\det S\le \det T$. 
This implies that $A(S,T)=0$ for $T$ with $\det T < \det S$. 
Therefore we have $a_g(T)=0$ for any $T\in {\mathcal T}_{r,r-1}$. 
This implies $g\in M_{k,{r-1}}^{p\text{-sing}}(\Gamma _0^{(r)}(N),\chi )$. 

Suppose that $g\not \equiv 0$ mod $p$. 
Then $g$ is mod $p$ singular of some $p$-rank $r'< r$. 
By Theorem \ref{Thm:Bo-Ki}, we have $2k-r\equiv 2k-r'\equiv 0$ mod $p-1$ and hence $r'\equiv r$ mod $p-1$. 
Since $r'<r<p-1$, this is impossible. 
This shows $g\equiv 0$ mod $p$. 
Taking into account that $F$ and $G$ are mod $p$ singular, we obtain 
that $G\equiv 0$ mod $p$. 
This completes the proof of (1) in Theorem \ref{Thm:general}. 
\qed
\begin{Rem}
The condition ``$n\ge r$'' was necessary to assure $\widetilde{g}\in \widetilde{M}_{k}(\Gamma _0^{(r)}(N),\chi)$. 
\end{Rem}

\section{Specification of levels}
\label{Sec:6}
\subsection{Proof of Theorem \ref{Thm:general} (2) for $m=1$}
In the situation of Theorem \ref{Thm:general} we have to
specify the levels of the $S\in \Lambda_r^+$ involved: 
To do so, we use Corollary \ref{Cor:Kita}, which follows from the modified
$q$-expansion principle and Kitaoka's formula. We observe that
$\widetilde{\theta^{(r)}_S}\in \widetilde{M}_k(\Gamma^{(r)}_0(N),\chi )$;
here we need the stronger condition $p\geq r+3$ (and also $n\ge r$).
Then Corollary \ref{Cor:Kita} shows that the level of $S$ is of the form requested
and also the statement about the nebentypus follows.
\qed
\vspace{0.3cm}

Now we prove Theorem \ref{Thm:r=2} in the remainder of this section. 
Namely, we discuss the case where the $p$-rank is $2$ and the weight is
the minimal one allowing mod $p$ singular modular forms,
which are not truly singular, i.e. $1+\frac{p-1}{2}$.

 
\subsection{Proof of Theorem \ref{Thm:r=2}}
In Theorem \ref{Thm:general}, we assumed $n\ge r$. 
In this subsection, we discuss the special case of $n=1$ and $r=2$; this is the simplest case where
the condition ``strongly mod $p$ singular'' is violated.
We use instead very special properties of binary quadratic forms.

\begin{Prop}
\label{PropA-2}
Let $p$ be an odd prime. 
Let $S\in \Lambda _2^+$ be of level $N=p^j  N'$ with $N'$ coprime to $p$. 
Suppose that $\theta _S^{(1)}\equiv \phi $ mod $p$ for some $\phi \in M_k(\Gamma _1)_{\Z_{(p)}}$. 
Then $N'=1$, i.e. ${\rm level}(S)=p^j$.  
\end{Prop}
\begin{Rem}
To include the case of $n=1$, we should look at the degree $1$ theta series $\theta _S^{(1)}$ and $\phi \in M_k(\Gamma _1)$. 
\end{Rem}

\begin{proof}
Seeking a contradiction we suppose $N'>1$.  
We apply Theorem \ref{Thm:Kita} of Kitaoka's original result for $n=1$ to $\theta _S^{(1)}$ with $M=\smat{*}{*}{p^j}{N'}$. 
Then we have 
\[\theta_S^{(1)}|M = \kappa\cdot \theta _{S'}^{(1)},  \]
where $S'$ is the same as in the proof of  Proposition \ref{PropA}.
Then $N'S'$ is half integral, $(N', {\rm cont}(N'S'))=1$, and ${\rm cont}(N'S')=p^\alpha $ when $p^\alpha \;|| \; {\rm cont}(S)$. 
Then $\frac{N'}{p^\alpha }S'\in \Lambda _2^+$ is primitive.
Therefore we can take a prime $l$ with $l\nmid N'$ such that 
\[A\left(S',\frac{l\cdot p^\alpha }{N'}\right)=A\left(\frac{N'S'}{p^\alpha },l\right)\not \equiv 0 \bmod{p}. \] 
However, since $\phi $ is of level $1$, we have also
\[A\left(S',\frac{l\cdot p^\alpha }{N'}\right)\equiv a_\phi \left(\frac{l\cdot p^\alpha }{N'}\right)=0 \bmod{p}. \] 
This contradicts and we get $N'=1$.  
\end{proof}
Hence ${\rm level}(S)$ is a power of $p$ in the situation of Proposition \ref{PropA-2}.   
Then $\det (2S)$ also should be a power of $p$. 
By the elementary divisor theorem, we can find $U$, $V\in {\rm GL}_2(\Z)$ such that 
\[U(2S)V=\mat{p^s}{0}{0}{p^{s+t}}. \]  
Then we have $U(2S)=\smat{p^s}{0}{0}{p^{s+t}}V^{-1}$ and hence 
\[U(2S){}^tU=p^s\mat{1}{0}{0}{p^t}V^{-1}{}^tU. \]
If we put $W:=V^{-1}{}^tU=\smat{w_1}{w_2}{w_3}{w_4}$, then we can write 
\[U(2S){}^tU=p^s\mat{w_1}{p^tw_3}{w_2}{p^tw_4}. \]
Since $U(2S){}^tU$ is symmetric, we have $w_2=p^tw_3$ and 
\[U(2S){}^tU=p^s\mat{w_1}{p^tw_3}{p^tw_3}{p^tw_4}. \]
From these argument, we may assume that $S$ is of this form.
In particular, if $\det (2S)$ is an odd power of $p$, by putting $b:=p^jw_3$, we may assume that $S$ is of the form
\[S=S(i,j)=p^i \begin{pmatrix}a & bp^{j+1} \\ bp^{j+1} & dp^{2j+1}\end{pmatrix}\]
with $adp-b^2p^2=p$.  

Now we try to show that high powers of $p$ in the level of $S$ imply high filtration weight of $\theta ^{(1)}_S$:
\begin{Prop}
\label{PropA-3}
Let $p$ be an odd prime and $S\in \Lambda _2^+$. 
Suppose that $\det(2S)$ is an odd power of $p$.  
Then we have
\[\omega (\theta ^{(1)}_S)\ge p^{i+j} \cdot \frac{p+1}{2},  \]
where $i$, $j$ are determined by 
\[S=S(i,j)=p^i \begin{pmatrix}a & bp^{j+1} \\ bp^{j+1} & dp^{2j+1}\end{pmatrix}\]
with $adp-b^2p^2=p$.  
\end{Prop}
Before proving this proposition, we confirm some notation and facts on filtration of modular forms mod $p$ for degree $1$ given by Serre \cite{Se} and Swwinerton-Dyer \cite{Sw}. 
Let $f=\sum _{n=0}^{\infty }a_f(n){\boldsymbol e}(nz)$ be a formal power series with $a_f(n)\in \Z_{(p)}$. 
We write $\omega (f)$ for $\omega _{1,1}^1(f)$. Let $V(p)$, $U(p)$ be the operators defined by
\begin{align*}
&f|V(p)=\sum _{n=0}^\infty a_f(n){\boldsymbol e}(pnz),\\
&f|U(p)=\sum _{n=0}^\infty a_f(pn){\boldsymbol e}(nz). 
\end{align*}
Then we have $\omega (f|V(p))=p\omega (f)$, $\omega (f|U(p))\le \omega (f)$. 

We put $\omega (S(i,j)):=\omega (\theta _{S(i,j)}^{(1)})$. 

\begin{proof}[Proof of Proposition \ref{PropA-3}]
Using the facts that $\theta _{S(i,j)}^{(1)}=\theta _{S(i-1,j)}^{(1)}|V(p)$ and $\omega (f|V(p))=p\omega (f)$, 
we have 
\[\omega (S(i,j))=p^i\omega (S(0,j)). \]
Furthermore, by comparing the representation numbers of $A(S(0,j),pn)$ and \\
$A(S(1,j-1),n)$, 
we can easily confirm that 
\[\theta ^{(1)}_{S(0,j)}| U(p)=\theta ^{(1)}_{S(1,j-1)}. \]
Since $\omega (f)\ge \omega (f|U(p))$ in general,  we have 
\[\omega (S(0,j))\ge \omega (S(1,j-1))=p\omega (S(0,j-1)). \]
This implies 
\[\omega (S(i,j))\ge p^{i+j}\omega (S(0,0))=p^{i+j}\omega (S(0,0))=p^{i+j}\cdot \frac{p+1}{2}.\]
Here the last  equality follows from $\omega (S(0,0))=\frac{p+1}{2}$ and this is due to Theorem \ref{Thm:Bo-Na} (or Serre's result in \cite{Se}). 
This completes the proof. 
\end{proof}
\begin{Cor}
\label{Cor1}
Let $p$ be a prime and $S\in \Lambda _2^+$. 
Suppose that $\theta _S^{(1)}\equiv \phi $ mod $p$ for some $\phi \in M_\frac{p+1}{2}(\Gamma _1)_{\Z_{(p)}}$. 
Then we have ${\rm level}(S)=p$.  
\end{Cor}
\begin{proof}
The assumption  $\theta _S^{(1)}\equiv \phi $ mod $p$ implies 
\begin{align}
\label{eq:filt_S}
\omega (S)=\omega (S(i,j))\le \frac{p+1}{2}.
\end{align}

On the other hand, the level of $S$ is a power of $p$ because of  Proposition \ref{PropA-2}. 
Then $\det (2S)$ is a power of $p$. 
Note here that, actually in this case, the power of $p$ is odd. 
Because, the quadratic character  $\chi _S$ should have nontrivial $p$-component, otherwise $\frac{p+1}{2}\equiv 1$ mod $p-1$ should hold and this is impossible. 
Hence we can assume that $S=S(i,j)$ as in Corollary \ref{PropA-3}. 
Then we have 
\[\omega (S)=\omega (S(i,j))\ge p^{i+j}\cdot \frac{p+1}{2}. \] 
Combing this with (\ref{eq:filt_S}), we have $i=j=0$. Namely, we obtain ${\rm level}(S)=p$. 
\end{proof}

We can now prove our result for the case of $p$-rank $2$. 
\begin{proof}[Proof of Theorem \ref{Thm:r=2}]
Let $\{T_1,\cdots,T_{h_p}\}\subset \Lambda _{2}/{\rm GL}_2(\Z)$ be a set of the representatives of ${\rm GL}_{2}(\Z)$-inequivalence classes 
of binary quadratic forms with level $p$. 
Note that $\theta ^{(n)}_{T_i}\in M_{1}(\Gamma ^{(n)}_0(p),(\frac{*}{p}))$, and there exists $G_i\in M_{\frac{p+1}{2}}(\Gamma _n)$ such that  $G_i \equiv \theta ^{(n)}_{T_i}$ mod $p$ by Theorem \ref{Thm:Bo-Na}.  
Then $G_i$ is mod $p$ singular, because $\theta ^{(n)}_{T_i}$ is true singular. 
Then we consider 
\begin{align*}
H:=F-\sum _{i=1}^{h_p} \frac{1}{2}a_F\begin{pmatrix}0 & 0 \\ 0 & T_i\end{pmatrix}G_i\in M_{\frac{p+1}{2}}(\Gamma _{n+2})
\end{align*}
This $H$ is a mod $p$ singular of some $p$-rank $r'\le 2$. 

Now we suppose that still $r'=2$. 
Then we can take $S\in \Lambda _2^+$ with $a_H\smat{0}{0}{0}{S}\not \equiv 0$ mod $p$ such that $\det S$ is minimal. 
By Corollary \ref{Cor:filt}, we have $\widetilde{\theta_S^{(1)}} \in \widetilde{M}_{\frac{p+1}{2}}(\Gamma _1)$. 
By Corollary \ref{Cor1}, we have ${\rm level}(S)=p$ and hence $S\sim T_j$ mod ${\rm GL}_2(\Z)$ for some $j$. 
However, from $a_H\smat{0}{0}{0}{T_i}\equiv 0$ mod $p$ for any $i$, we have $a_H\smat{0}{0}{0}{S}\equiv 0$ mod $p$. 
This is a contradiction. 

Therefore we have $r'\le 1$ or $H\equiv 0$ mod $p$.   
If $H\not \equiv 0$ mod $p$, then we have $p+1-r'\equiv 0$ mod $p-1$ because of Theorem \ref{Thm:Bo-Ki}. 
This is impossible. This implies $H\equiv 0$ mod $p$. 
This shows the main statement of Theorem \ref{Thm:r=2}.

To prove the final statement concerning $\Phi^{n+1}(F)\not \equiv 0$ mod $p$ it is sufficient to
assure the mod $p$ linear independence of the binary theta series in question: 
Let $\{T_1,\cdots,T_{h_p}\}$ be as above. 
As stated in Kani \cite{Kani}, $\theta _{T_j}^{(1)}$ ($j=1, \cdots, h_p$) are linearly independent over $\C$. 
Now we can prove that $\theta _{T_j}^{(1)}$ ($j=1, \cdots, h_p$) are linearly independent over $\F_p$:  
Let $\sum _{i=1}^{h_p} c_i \theta _{T_j}^{(1)}\equiv 0 \bmod{p}$. 
Then, for any $n\ge 0$, we have 
\[\sum _{i=1}^{h_p} c_i A(T_i,n)\equiv 0 \bmod{p}.\] 
For each $T_i$, we can find infinitely many primes $l$ such that $A(T_i,l)>0$. 
Then we have $A(T_j,l)=0$ for any $j\neq i$. Since $A(T_i,l)=2,4,6\not \equiv 0$ mod $p$, we have $c_j \equiv 0$ mod $p$. 
Hence we have $c_i\equiv 0$ mod $p$. This shows $c_i\equiv 0$ mod $p$ for any $i$ with $1\le i\le h_p$. 
Therefore $\theta _{T_j}^{(1)}$ ($j=1$, $\cdots$, $h_p$) are linearly independent over $\F_p$. 
This completes the proof of Theorem \ref{Thm:r=2}. 
\end{proof}

\section{From mod $p$ to mod $p^m$} 
\label{Sec:7}
In this section, using induction, we extend Theorem \ref{Thm:general} (proved for $m=1$ in the previous sections) to the case of general $m$.
The main reason why we preferred to give a mod $p$ version first and then
extend to arbitrary powers is the technical difficulty which one encounters,
when one tries to give a mod $p^m$ analogue of our abstract Sturm bound.
We avoid such a  problem in this way. 

We start with proving the following weaker version of  Theorem \ref{Thm:general}.
\begin{Thm}[Weaker version of Theorem \ref{Thm:general}]
\label{Thm:weak}
Let $n$, $k$, $N$ be positive integers, $r$ an even integer with $n\ge r$. 
Let $p$ be a prime with $p>r+1$ and $\chi $ a quadratic
Dirichlet character mod $N$ with $\chi (-1)=(-1)^k$.  
Suppose that $F\in M_{k}(\Gamma _0^{(n+r)}(N),\chi )_{\Z_{(p)}}$ is
mod $p^m$ singular of $p$-rank $r$. 
Then there are finitely many $S\in \Lambda _r^+$ of level dividing $p^eN$ with
$\widetilde{\theta _{S}^{(n)}}\in \widetilde{M}_k(\Gamma ^{(n)}_0(N),\chi )$ 
(only mod $p$) such that 
\begin{align} 
\label{xy} 
F\equiv \sum_S c_S\theta_S^{(n+r)} \bmod{p^m}
  \quad (c_S\in \Z_{(p)}).
\end{align}
Moreover, all $S$ involved satisfy $\chi='\chi_S$. 
\end{Thm}
\begin{proof}[Proof of Theorem \ref{Thm:weak}]
We prove the statement by induction on $m$.
We have already proved the statement for $m=1$.  
Suppose that the statements (1), (2) in Theorem \ref{Thm:general} are true for any $m$ with $m<m_0$. 
We consider the case $m=m_0$. 
Note that $2k-r\equiv 0$ mod $(p-1)p^{m_0-1}$ by Theorem \ref{Thm:Bo-Ki}. 
Therefore we can write $k=\frac{r}{2}+t \cdot \frac{p-1}{2}p^{m_0-1}$. 

By the assumption, $F$ is mod $p^{m_0-1}$ singular (because of mod $p^{m_0}$ singular) and therefore by induction hypothesis, 
we have 
\[F\equiv \sum _{\substack{S\in \Lambda _r^+\\ {\text{level}(S)}\mid p^{e_1}N}}c_S\theta _S^{(n+r)} \bmod{p^{m_0-1}}\quad \text{with}\quad c_S\in \Z_{(p)}, \]
for some $e_1$, 
and all $S$ involved satisfy $\chi =' \chi_S$.  
 
We take ${\mathcal E}\in M_{\frac{p-1}{2}}(\Gamma _0^{(n+r)}(p),(\frac{*}{p}))_{\Z_{(p)}}$ from \cite{Bo-Na:2007} such that ${\mathcal E}\equiv 1$ mod $p$ and  
put $G:=\sum _{S}c_S\theta _S^{(n+r)}\cdot {\mathcal E}^{tp^{m_0-1}}$.  
Note here that $F$, $G\in M_{k}(\Gamma _0^{(n+r)}(p^{e_1}N),\chi )_{\Z_{(p)}}$ 
because of $\chi ='\chi _S$. 
Consider $H:=\frac{1}{p^{m_0-1}}(F-G)\in M_{k}(\Gamma _0^{(n+r)}(p^{e_1}N),\chi )_{\Z_{(p)}}$.   
Then for any $T$, we have 
\[a_F(T)=a_G(T)+p^{m_0-1}a_H(T). \]
Since $a_F(T)\equiv a_G(T)\equiv 0$ mod $p^{m_0}$ for any $T$ with ${\rm rank}(T)>r$, we have $a_H(T)\equiv 0$ mod $p$ for any $T$ with ${\rm rank}(T)>r$. 
This means that $H$ is mod $p$ singular of some $p$-rank $r'\le r$. 
By $2k-r'\equiv 2k-r \equiv 0$ mod $p-1$ and $0\le r'\le r \le p-1$, we have $r'=r$. 
Therefore $a_H(T)\not \equiv 0$ mod $p$ for some $T$ with ${\rm rank}(T)=r$. 
Again by our theorem for $m=1$, we have 
\[H\equiv \sum _{\substack{R\in \Lambda _r^+\\ {\text{level}(R)}\mid p^{e_2}N}}d_R\theta _R^{(n+r)} \bmod{p}\quad \text{with}\quad d_R\in \Z_{(p)}, \] 
and all $R$ involved satisfy $\chi =' \chi_R$. 
Noting $G:=\sum _{S}c_S\theta _S^{(n+r)}\cdot {\mathcal E}^{tp^{m_0-1}}\equiv \sum _{S}c_S\theta _S^{(n+r)}$ mod $p^{m_0}$, we have
\begin{align*}
F&\equiv G+p^{m_0-1}H\\
&\equiv \sum _{\substack{S\in \Lambda _r^+\\ {\text{level}(S)}\mid p^{e_1}N}}
c_S\theta _S^{(n+r)}+p^{m_0-1}\sum _{\substack{R\in \Lambda _r^+\\ {\text{level}(R)}\mid p^{e_2}N}}d_R\theta _R^{(n+r)}  \bmod{p^{m_0}}. 
\end{align*} 
This completes the proof of Theorem \ref{Thm:weak}.  
\end{proof}

\begin{proof}[Completion of the proof of Theorem \ref{Thm:general}]
  We first claim that the coefficients $c_S$ in (\ref{xy})
  coincide mod $p^m$ with $\frac{a^*(S)}{\epsilon(S)}$ provided that the
  summation in (\ref{xy}) is restricted to pairwise ${\rm GL}_r({\mathbb Z})$-inequivalent elements of  $\Lambda^+_r$.
  To see this, we take an arbitrary $S_0\in \Lambda_r^+$.
  Then (using freely the notaion from Section \ref{Sec:4})
  $$a(S_0)\equiv \sum_S c_S A(S,S_0)\bmod p^m$$
  holds and passing to primitive Fourier coefficients gives
  $$a^*(S_0)\equiv \sum_S c_S A^*(S,S_0)\bmod p^m.$$
  Now $A^*(S,S_0)=\epsilon(S_0)$ if $S$ and $S_0$ are ${\rm GL}_r(\Z)$-equivalent and zero otherwise; this proves the claim from above.

 To complete the proof of Theorem \ref{Thm:general},
 we have to prove
 $\widetilde{\theta _{S}^{(n)}}\in \widetilde{M}_k
 (\Gamma ^{(n)}_0(N),\chi )_{p^{m-\nu}}$; this follows now directly
 from Proposition \ref{Prop:filt}.
This completes the proof of Theorem \ref{Thm:general}.
\end{proof}

\section{Appendix A: Modified $q$-expansion principle}
\label{Sec:8}
To specify the level of theta series in our theorems, we need some control over the behaviour of
congruences of modular forms, when we switch to other cusps.
The ``$q$-expansion principles" available in the literature do not exactly provide the information
necessary for our purpose.

We fix a prime $p$ and  a natural number $N$ coprime to $p$.
For $m\geq 0$ let
$R_m$ be a subring of ${\mathbb C}$ containing 
${\mathbb Z}[
e^{\frac{2\pi i}{N\cdot p^m}}, 
e^{\frac{2\pi i}{3}},
\frac{1}{6},\frac{1}{N}]
$. 

\begin{Thm}
\label{Thm:q-exp}
Let $p$ be a prime with $p\ge n+3$ and $N$ a positive integer with $p\nmid N$. 
Let $f\in M_k(\Gamma _0^{(n)}(p^m)\cap \Gamma^{(n)}(N))$. 
Then  for all $\gamma \in \Gamma _0^{(n)}(p^m)$ we have
$$f\in    M_k(\Gamma _0^{(n)}(p^m)\cap \Gamma^{(n)}(N))_{R_m} \iff
f|_k \gamma\in  M_k(\Gamma _0^{(n)}(p^m)\cap \Gamma^{(n)}(N))_{R_m}.$$
\end{Thm}
The case $m=0$ is the statement of the usual $q$-expansion principle in the formulation of
Pitale et al. \cite{PiScSa} proved by Katz \cite{Kat} for $n=1$ and Ichikawa \cite{Ich} for $n>1$. 
The main point in our version is that $f$ and $f| \gamma$ share the same $p$-integrality property.
Note that one only needs to prove one direction of the theorem.
The aim of Appendix A is to prove this property.
\subsection{Proof for $m=1$}
We will do induction on $m$.
We need the existence of a modular form ${\mathcal E}\in M_{l}(\Gamma _0^{(n)}(p))_{\Z_{(p)}}$ such that 
\begin{align*}
&{\mathcal E}\equiv 1 \bmod{p}\\
&{\mathcal E}| \omega _j \equiv 0 \bmod{p} \quad (1\le j \le n ). 
\end{align*}
Here $\omega _j$ can be any element $\smat{*}{*}{C}{D}\in \Gamma _n$ with $C$ being of rank $i$ in $M_{n}(\F_p)$.
The existence of such ${\mathcal E}$ was explained in \cite{Bo2017}, under the condition $p\ge n+3$. 
We mention that the congruence condition of ${\mathcal E}$ can be rephrased as saying that for any $\gamma \in \Gamma _n$, we have 
\[{\mathcal E}| \gamma \equiv 1 \bmod{p} \quad \iff \quad \gamma \in \Gamma _0^{(n)}(p)\]
and ${\mathcal E}| \gamma \equiv 0$ mod $p$ for $\gamma \not \in \Gamma _0^{(n)}(p)$. 

We prove the statement for $m=1$. 
We chose an appropriate power $\alpha $ of ${\mathcal E}$ and consider a trace from level $\Gamma _0^{(n)}(p)\cap \Gamma ^{(n)}(N)$ to level $\Gamma ^{(n)}(N)$:
\[F=f{\mathcal E}^\alpha +\sum _{j\bmod{p}}(f{\mathcal E}^{\alpha })|\delta _j\]
where $\delta _j$ are representatives of the left cosets not in $\Gamma _0^{(n)}(p)\cap \Gamma ^{(n)}(N)$,
in particular, ${\mathcal E}|\delta _j\equiv 0$ mod $p$. 
Applying $\gamma $, we see that $\delta _j \cdot \gamma  \not \in \Gamma _0^{(n)}(p)$ and hence $(f {\mathcal E}^\alpha )|\delta _j \cdot \gamma \equiv 0$ mod $p$. 
We observe that $F$ is of level $\Gamma ^{(n)}(N)$ and hence we may apply the ordinary $q$-expansion principle: 
Then we have 
\[F|\gamma \equiv (f|\gamma ) {\mathcal E}^\alpha \equiv f|\gamma \bmod{p}. \]
In particular $F|\gamma $ is $p$-integral and hence $f|\gamma $ is also $p$-integral. 
This completes the proof for $m=1$. 
\qed

\subsection{Proof for $m\ge 2$}
We prove it for $m\ge 2$ by induction on the power $m$.
We suppose that the statement is true for the level $\Gamma _0^{(n)}(p^{m-1})$. 
We put 
\[{\mathcal F}:={\mathcal E}(p^{m-1}z)=p^{-(m-1)\frac{nl}{2}}{\mathcal E}| \begin{pmatrix}p^{m-1} \cdot 1_n & 0_n \\ 0_n & 1_n\end{pmatrix} \in M_{l}(\Gamma _0^{(n)}(p^{m}))_{\Z_{(p)}}. \]
Here  $l$ is the weight of ${\mathcal E}$. 
Taking some positive integer $\alpha $, we consider $f{\mathcal F}^\alpha $ and the trace of this from $\Gamma _0^{(n)}(p^m)\cap \Gamma ^{(n)}(N)$ to $\Gamma _0^{(n)}(p^{m-1})\cap \Gamma ^{(n)}(N)$; 
\begin{align}
\label{Eq:tr}
F={\rm tr}(f{\mathcal F}^\alpha )&=f{\mathcal F}^\alpha +\sum _{w}^{p-1}(f{\mathcal F}^\alpha )|\gamma _w \\ \nonumber
&=f{\mathcal F}^\alpha +\sum _{j=1}^{p-1}(f|\gamma _w)({\mathcal F}|\gamma _w)^\alpha , 
\end{align}
where $\gamma _w=\smat{1_n}{0_n}{p^{m-1}N_w}{1_n}$ and $w$ runs over a complete set of representatives of 
integral symmetric matrices of size $n$ mod $p$ except the trivial coset $p\cdot {\rm Sym}_n(\Z)$. 
Now we prove that $(f|\gamma _w )({\mathcal F}|\gamma _w)^\alpha \equiv 0$ mod $p$. 
By a direct calculation, we have 
\begin{align}
\label{Eq:mt}
\begin{pmatrix}
p^{m-1}\cdot 1_n & 0_n \\ 0_n & 1_n
\end{pmatrix}
\begin{pmatrix}
1_n & 0_n \\ p^{m-1}Nw  & 1_n
\end{pmatrix}
&=\begin{pmatrix}
1_n & 0_n \\ Nw & 1_n
\end{pmatrix}
\begin{pmatrix}
p^{m-1}\cdot 1_n & 0_n \\ 0_n  & 1_n
\end{pmatrix}.
\end{align} 
This implies 
\begin{align*}
{\mathcal F}| 
\begin{pmatrix}
1_n & 0_n \\ p^{m-1}Nw  & 1_n
\end{pmatrix}
&=p^{-(m-1)\frac{nl}{2}}
{\mathcal E}|\begin{pmatrix}
p^{m-1}\cdot 1_n & 0_n \\ 0_n & 1_n
\end{pmatrix}
\begin{pmatrix}
1_n & 0_n \\ p^{m-1}Nw  & 1_n
\end{pmatrix}\\
&=p^{-(m-1)\frac{nl}{2}}{\mathcal E}|\begin{pmatrix}
1_n & 0_n \\ Nw & 1_n
\end{pmatrix}
\begin{pmatrix}
p^{m-1}\cdot 1_n & 0_n \\ 0_n  & 1_n
\end{pmatrix}.
\end{align*}
Since ${\mathcal E}|\smat{1_n}{0_n}{Nw}{1_n}\equiv 0$ mod $p$ and we can put ${\mathcal E}|\smat{1_n}{0_n}{Nw}{1_n}=pX$ for some $p$-integral modular form $X=X_w$. 
Then we have 
\begin{align*}
{\mathcal F}| 
\begin{pmatrix}
1_n & 0_n \\ p^{m-1}Nw  & 1_n
\end{pmatrix}&=p^{-(m-1)\frac{nl}{2}}{\mathcal E}
|\begin{pmatrix}
1_n & 0_n \\ Nw & 1_n
\end{pmatrix}
\begin{pmatrix}
p^{m-1}\cdot 1_n & 0_n \\ 0_n  & 1_n
\end{pmatrix}\\
&=p^{-(m-1)\frac{nl}{2}} (p X) |
\begin{pmatrix}
p^{m-1}\cdot 1_n & 0_n \\ 0_n  & 1_n
\end{pmatrix}\\
&=pX(p^{m-1}z)\equiv 0 \bmod{p}.
\end{align*}
Therefore $(f|\gamma _w )({\mathcal F}|\gamma _w)^\alpha \equiv 0$ mod $p$ follows if $\alpha $ is large. 
By the assumption that $f$ is $p$-integral, $F$ is also $p$-integral. 

Now let $\delta \in \Gamma _0^{(n)}(p^m)$, we may apply induction because  $F$ is of level $\Gamma _0^{(n)}(p^{m-1}N)$.  
We start form (\ref{Eq:tr}) for $F$ and apply $\delta $ to it:
If $f$ is $p$-integral, then $F$ is also $p$-integral, and then $F|\delta $ is also $p$-integral. 
By (\ref{Eq:tr}), we have 
\begin{align*}
F|\delta =(f|\delta ){\mathcal F}^\alpha +\sum _{j=1}^{p-1}(f{\mathcal F}^\alpha )|\gamma _w \delta 
\end{align*} 
The first term on the right hand side of (\ref{Eq:tr}) is congruent mod $p$ to $f|\delta $. 
We prove $(f{\mathcal F}^{\alpha })|\gamma _w\delta \equiv 0$ mod $p$. 
To prove this, we must multiply (\ref{Eq:mt}) from the right by $\delta =\smat{a}{b}{p^mc}{d}$. 
We get 
\[
\begin{pmatrix}
p^{m-1}\cdot 1_n & 0_n \\ 0_n & 1_n 
\end{pmatrix}
\gamma _w \delta 
=\begin{pmatrix}
a & p^{m-1}b \\ wNa+pc & wNp^{m-1}b+d
\end{pmatrix}
\begin{pmatrix}
p^{m-1}\cdot 1_n & 0_n \\ 0_n & 1_n 
\end{pmatrix}. 
\] 
Since $p\nmid wNa$ and by a similar argument as above, we obtain ${\mathcal F}|\gamma _w \delta \equiv  0$ mod $p$. 
This shows $(f{\mathcal F}^{\alpha })|\gamma _w\delta \equiv 0$ mod $p$ for large $\alpha $.
This implies that $F|\delta \equiv f|\delta $ mod $p$ and hence $f|\delta $ is $p$-integral. 
\qed

\begin{Rem} 
A minor variant of the proof above applies if $f$ is of  ``quadratic nebentypus mod $p$",
i.e. $f$ satisfies (under the same conditions as in the theorem above)
$f|_k \gamma=(\frac{\det(d)}{p}) \cdot f$ for  
all $\gamma\in \Gamma_0^{(n)}(p^m)\cap\Gamma^{(n)}(N)$
with the character $(\frac{*}{p})$.
The proof above needs to be modified only for the step ``$m=1$", where one has to
use a function ${\mathcal E}$ with nebentypus $(\frac{*}{p})$.
\end{Rem}
\section{Appendix B: Kitaoka's formula}
\label{Sec:9} 
In this Appendix, we generalize Kitaoka's theorem to higher degree to specify the level of theta series in our theorem. 
We then apply this in practice and discuss levels.
\subsection{Generalization of Kitaoka's formula}
We freely switch here between the language of (positive integral) quadratic forms $S$ (or integral symmetric matrices) and the language of lattices $L$ in an euclidean space $({\mathbb R}^m,\left<\ ,\ \right>)$. 
The aim is to get a degree $n$ version of a result of Kitaoka; 
we state it here only in the version sufficient for our application and allowing a rather short proof. 
A more sophisticated version generalizing the computation of Kitaoka would lead to a more general statement. 
\begin{Thm}
\label{Thm:Kita}
Let $L$ be an even integral positive definite lattice of rank $m=2k$ and level $N$. 
Let $d$ be a positive divisor of $N$ with $(d,\frac{N}{d})=1$, and $c:=\frac{N}{d}$. 
Choose integers $a$, $b$ such that $M=\smat{a}{b}{c}{d}\in {\rm SL}_2(\Z)$ and denote by $\frak{M}$ the $n$-fold 
diagonal embedding of $M$ into ${\rm Sp}_n(\Z)$, namely, 
$\frak{M}:=\smat{a\cdot 1_n}{b\cdot 1_n}{c\cdot 1_n}{d\cdot 1_n}$.

Then, with a suitable constant $\kappa $ 
\begin{align}
\label{Eq:Kita}
\theta _L^{(n)}| _k\frak{M}=\kappa \cdot \theta _{L'}^{(n)}, 
\end{align}
where $L'$ is a lattice characterized by 
\[L'\otimes \Z_p =
\begin{cases}
L\otimes \Z_p\ &\text{if}\ p\nmid d, \\
(L\otimes \Z_p)^*\ &\text{if}\ p\mid d, 
\end{cases}\]
and $*$ denotes the dual. 
\end{Thm}
Before we begin the proof, we need to prepare some more notation and recall the facts: 
Kitaoka \cite{Kita} proved the degree one version of this statement for a more general type of theta series
$$\theta^{(1)}_L(q)(z)=\sum_{{\mathfrak x}\in L} q({\mathfrak x}) e^{2\pi i \left<{\mathfrak x},{\mathfrak x}\right>\cdot z}$$
 allowing homogeneous harmonic polynomials $q$ on ${\mathbb C}^m$ as coefficients in the theta series. 
An inspection of his proof shows that the constants of proportionality  do NOT depend on those harmonic polynomials. 

Ibukiyama \cite{Ibu} investigated holomorphic differential operators ${\mathcal D}$ on $\hh_n$, which are polynomials in the partial holomorphic derivatives, 
evaluated in ${\rm  diag}(z_1,\cdots , z_n)$; Such a ${\mathcal D}$ maps holomorhic functions $F$ on ${\mathbb H}_n$ to holomorphic
functions ${\mathcal D}(F)$ defined on ${\mathbb H}_1^n$
such that, with the notation above 
\[{\mathcal D}(F|_k {\frak M})={\mathcal D}(F)|^{(z_1)} _{k_1}M\cdots |^{(z_n)} _{k_n}M \] 
holds for arbitrary $M\in {\rm SL}_2({\mathbb R})$ with certain weights $k_i\ge k$. 
The upper index $z_i$ at the slash-operator indicates that $|_{k_i}M$ should be applied w.r.t. the variable $z_i$.

He also showed that from all the ${\mathcal D}(F)$ one can recover the Taylor  expansion of $F$ with respect to the non-diagonal coefficients of $Z=(z_{ij})\in \hh_n$. 
To show the theorem, it is then enough, to show the equality of both sides after applying all such differential operators. 

 To do this we need some more notation:    
 For a fixed differential operator ${\mathcal D}$ of Ibukiyama type and $w_1, \cdots, w_n\in \C ^m$ changing the automorphy 
 factor from $\det ^k$ to $k_1,\cdots,k_n$ in the 
 variables $z_1,\cdots ,z_n$ we define
  \[{\mathcal D}e^{2\pi i \sum \left<w_i,w_j\right>z_{ij}}=P(w_1,\cdots ,w_n )e^{2\pi i \sum \left<w_i,w_i\right>z_{i}}\]
  with $z_i:=z_{ii}$. 
Then $P$ is a harmonic polynomial in each of the variables $w_i$ and we may write it as a finite sum of pure tensors:
\[P(w_1, \cdots ,w_n)=\sum _{t}q_{1,t}(w_1)\cdots q_{n,t}(w_n),\]
where the $q_{i,t}$ are harmonic polynomials in the variables $w_i$. 

\begin{proof}[Proof of Theorem \ref{Thm:Kita}]
Applying ${\mathcal D}$ to the right hand side of (\ref{Eq:Kita}) gives 
\[\kappa \cdot \sum _t \theta _{L'}^{(1)}(q_{1,t})(z_1)\cdots \theta _{L'}^{(1)}(q_{n,t})(z_n).\]
Applying ${\mathcal D}$ to the left hand side gives 
\[\sum _t \theta _{L}^{(1)}(q_{1,t})(z_1)\cdots \theta _{L}^{(1)}(q_{n,t})(z_n)|_{k_1}M\cdots |_{k_n}M.\]
Now one has to use Kitaoka's result in degree $1$ to get our assertion. 
\end{proof}

\subsection{Application}
Kitaoka's formula allows us to exclude certain congruences between modular forms and theta series of higher level (as far as the prime to $p$ components of the levels are concerned): 
\begin{Prop}
\label{PropA}
Let $n$ be an even positive integer and $p$ a prime with $p\ge n+1$.  
Let $N$, $N'$ be coprime to $p$ and $S\in \Lambda_n^+$ with level $p^aN$.
Assume that
$$p^bN'\mid p^aN\quad\mbox{and}\quad \phi \equiv \theta^{(n)}_S\bmod p$$
for some $\phi\in M_k(\Gamma_0^{(n)}(p^bN'),\chi )_{{\mathbb Z}_{(p)}}$ with $\chi ^2=1$. 
Then $N=N'$ and $\chi _N=(\chi_S)_N$.

\end{Prop}
\begin{proof}
We want to apply the modified $q$-expansion principle. 
To do so, we must modify $\phi$ and $\theta^{(n)}_S$  to arrive at the same weights and same (quadratic) $p$-component of nebentypus: 
We recall from B\"ocherer-Nagaoka \cite{Bo-Na:2007} that 
there exist degree $n$ modular forms $\mathcal E$
  for $\Gamma_0^{(n)}(p)$ of weight $\frac{p-1}{2}$ of nontrivial quadratic nebentypus.
Then Proposition \ref{weightcongruence} allows us to choose $t\in {\mathbb N}$ such that $\theta^{(n)}_S\cdot {\mathcal E}^t$ and $\phi$ have the same weight
  and the same $p$-component of nebentypus (this holds if $k\geq \frac{n}{2}$;
  the modification for the case of the opposite inequality is obvious).
  Then we can indeed apply Theorem \ref{Thm:q-exp} to 
  $\frac{1}{p}\left( \phi-\theta^{(n)}_S\cdot {\mathcal E}^t\right)$.

  We may enlarge $b$ (if necessary) to get $a=b$.
  Assume that there is a prime $q$ different from $p$ with $q^r||N$ and $q^s || N'$ with $s<r$. 
We apply Kitaoka's theorem to $\theta _S^{(n)}$ with 
\[M=\mat{*}{*}{p^a\cdot \frac{N}{q^r}}{q^r}. \]
By the modified $q$-expansion principle Theorem \ref{Thm:q-exp}, we
have $\theta^{(n)}_S| \mathfrak{M}\equiv \pm \phi | \mathfrak{M}$ mod $p$,
where the $\pm$ comes in from applying $M$ to ${\mathcal E}^t$.

Let $L$ correspond to $S$ and correspond $L'$ to $S'$. 
Then $S'$ is a rational symmetric matrix with
$q^r$ appearing in its denominator. 
The Fourier expansion of
$\theta _S^{(n)}| {\frak M}=\kappa \cdot \theta _{S'}^{(n)}$
has a $S'$th Fourier coefficient which is 
\[\kappa \cdot A(S',S').\]
By the condition $p>n+1 $, we see that $A(S',S')\not \equiv 0$ mod $p$. 
On the other hand, $\phi| {\mathfrak M}$ does not have a nonzero 
Fourier coefficient at 
$S'$, because the $q$-part of the width of the cusp
${\mathfrak M}$ for $\Gamma_0^{(n)}(p^aN')$ is $q^s$ with $s<r$.
Therefore, $\kappa\cdot A(S',S')$
and hence $\kappa$ must be divisible by $p$, i.e. $\kappa\in p\cdot R_{p^a}$.
But the modified$q$-expansion principle
tells us that
$\theta^{(n)}_{S'}| {\mathfrak M}^{-1}= \frac{1}{\kappa}\cdot \theta^{(n)}_S$
is $p$-integral, in particular $\frac{1}{\kappa}\in R_{p^a}$. 
This is a contradiction, 
and we obtain $N'=N$.

The statement about the $N$-components of nebentypus
characters follows by applying the
modified $q$-expansion principle Theorem \ref{Thm:q-exp} with
any integral matrix $\smat{A}{B}{C}{D}$ satisfying $C\equiv 0_n$ mod $p^rN$
and $\det(D)\equiv 1\bmod p$.
\end{proof}
\begin{Cor}
\label{Cor:Kita}
Let $n$ be a positive even integer and $p$ a prime with $p>n+3$. Let $S\in \Lambda^+_n$.
Assume that $\theta _S^{(n)}\equiv \phi $ mod $p$ for some $\phi \in M_k(\Gamma_0^{(n)}(N),\chi )$ with $\chi ^2=1$.     
Then ${\rm level}(S)$ is of the form
``$p$-power $\times N'$'' with some $N'\mid N$
and $\chi =' \chi_S $.
\end{Cor}

\section*{Acknowledgment}
The authors would like to thank Professors G. Nebe and R. Schulze-Pillot for helpful advices, 
especially on the order of automorphism groups for quadratic forms and on the generalization of Kitaoka's transformation formula.
They would also like to thank Professor T. Yamauchi for giving them information on $q$-expansion principles.  
This work was supported by JSPS KAKENHI Grant Number 22K03259.


\providecommand{\bysame}{\leavevmode\hbox to3em{\hrulefill}\thinspace}
\providecommand{\MR}{\relax\ifhmode\unskip\space\fi MR }
\providecommand{\MRhref}[2]{%
  \href{http://www.ams.org/mathscinet-getitem?mr=#1}{#2}
}
\providecommand{\href}[2]{#2}

\begin{flushleft}
Siegfried B\"ocherer\\
Kunzenhof 4B \\
79177 Freiburg, Germany \\
Email: siegfried.boecherer@uni-mannheim.de
\end{flushleft}

\begin{flushleft}
  Toshiyuki Kikuta\\
  Faculty of Information Engineering\\
  Department of Information and Systems Engineering\\
  Fukuoka Institute of Technology\\
  3-30-1 Wajiro-higashi, Higashi-ku, Fukuoka 811-0295, Japan\\
  E-mail: kikuta@fit.ac.jp
\end{flushleft}

\end{document}